\documentclass{article}  

\usepackage{amsmath}
\usepackage{amsfonts}
\usepackage{amssymb}
\usepackage{amsthm}
\usepackage{mathtools}
\usepackage{mathrsfs}
\usepackage{mathpazo}

\usepackage{graphicx}
\usepackage{enumerate,color}
\usepackage{microtype}
\usepackage{sectsty}
\allsectionsfont{\sffamily}
\usepackage[parfill]{parskip}

\usepackage{glossaries}
\newacronym{lqr}{LQR}{linear-quadratic-regulator}

\usepackage{hyperref}
\usepackage[nameinlink]{cleveref}

\definecolor{lightblue}{rgb}{0.54, 0.81, 0.94}
\hypersetup{
    colorlinks = true,
    linkcolor = lightblue,
    citecolor = lightblue,
    urlcolor = lightblue,
    }
    
\crefformat{equation}{(#2#1#3)}
\crefmultiformat{equation}{(#2#1#3)}%
{ and~(#2#1#3)}{, (#2#1#3)}{ and~(#2#1#3)}

\newcommand{\realR}{\mathbb{R}}
\newcommand{\itemi}{($i$)}
\newcommand{\ii}{($ii$)}
\newcommand{\iii}{($iii$)}
\newcommand{\abs}[1]{\ensuremath{\left\vert{#1}\right\vert}}
\newcommand{\cfunof}[1]{\ensuremath{\left\{#1\right\}}}
\DeclareMathOperator{\diag}{diag}

\newcommand{\funof}[1]{\ensuremath{\left(#1\right)}}

\newcommand{\hermR}{\mathbb{S}}

\DeclareMathOperator{\tr}{tr}
\DeclareMathOperator{\vc}{vec}

\newtheorem{thm}{Theorem}
\newtheorem{lem}{Lemma}
\newtheorem{defn}{Definition}

\begin{document}

\title{Optimal Control on Positive Cones}

\author{Richard Pates and Anders Rantzer\thanks{The authors are with the Department of Automatic Control LTH, Lund
    University, Box 118, SE-221 00 Lund, Sweden (e-mail: richard.pates/anders.rantzer@control.lth.se).The authors are supported by the Excellence Center ELLIIT as well as the Wallenberg AI, Autonomous Systems and Software Program (WASP) and the European Research Council (Advanced Grant 834142).}}

\maketitle

\begin{abstract}
\noindent{}An optimal control problem on finite-dimensional positive cones is stated. Under a critical assumption on the cone, the corresponding Bellman equation is satisfied by a linear function, which can be computed by convex optimization. A separate theorem relates the assumption on the cone to the existence of minimal elements in certain subsets of the dual cone. Three special cases are derived as examples. The first one, where the positive cone is the set of positive semi-definite matrices, reduces to standard linear quadratic control. The second one, where the positive cone is a polyhedron, reduces to a recent result on optimal control of positive systems. The third special case corresponds to linear quadratic control with additional structure, such as spatial invariance.
\end{abstract}


\section*{Notation}

A cone $\mathcal{P}\subseteq{}\realR^n$ is said to be proper if it is convex, closed, pointed and has a non-empty interior. The dual cone of a proper cone $\mathcal{P}$ is also proper, and given by 
\[
\mathcal{P}^*=\cfunof{\lambda{}:\lambda{}^\top{}x\geq{}0\text{ for all }x\in\mathcal{P}}.
\]
All further non-standard notation will be introduced as it is used.

\section{Introduction}

This paper is concerned with optimal control of the form
  \begin{align}
    \text{Minimize}\quad&\sum_{t=0}^\infty[s^\top x(t)+r^\top u(t)]
    \text{ over }\{u(t)\}_{t=0}^\infty\label{eqn:problem}\\ 
    \text{subject to}\quad& x(t+1)=Ax(t)+Bu(t)\notag\\
    &(x(t),u(t))\in\mathcal{P},\quad x(0)=x_0.\notag
  \end{align}
  where $\mathcal{P}$ is a proper cone in $\realR^n$. Assumptions are made to make sure that the objective function is non-negative and thus finite if and only if the Bellman equation
  \begin{align*}
    J^*(x)&=\min_{(x(t),u(t))\in\mathcal{P}}
    \left[s^\top x+r^\top u+J^*(Ax+Bu)\right]
  \end{align*}
  has a non-negative solution $J^*$. The minimizing $u$ gives an optimal control policy. See for example \cite[section~3.1]{bertsekas2007dynamic}.
  
  The setting above turns out be one of the rare situations when the Bellman equation has an explicit solution. Otherwise, the most well known case from the literature is probably linear dynamics and quadratic cost, introduced in the pioneering work of \cite{kalman1960contributions}. Another case of fundamental importance is when both $x$ and $u$ are restricted to finite sets. This setting leads to shortest path problems on graphs, with an extensive literature dating back to the work of \cite{dijkstra1959note}. Recently, a third class of problems with explicit solutions to the Bellman equation was introduced in the context of linear systems on the positive orthant in $\realR^n$ (the so called positive systems), see \cite{rantzer2022explicit,li2023dynamic}. A closely related case with bilinear dynamics is treated in \cite{blanchini2023optimal}.
  
The purpose of this note is to show that several previous instances where the Bellman equation can be solved explicitly are in fact special cases of a more general class of optimal control problems stated in terms of positive cones.

\newpage{}

\section{Main Result}

\begin{thm}\label{thm:main}
  Let $\mathcal{P}$ be a proper cone in $\realR^{n+m}$ with dual cone $\mathcal{P}^*$. Given
  $A\in\realR^{n\times n}$, $B\in\realR^{n\times m}$, suppose that
  for every $(x,u)\in\mathcal{P}$ there exists $v$ such that $(Ax+Bu,v)\in\mathcal{P}$. Let $(s,r)$ be an interior point of $\mathcal{P}^*$. Moreover, assume the existence of 
  $\phi:\realR^m\to\realR^n$ such that
  \begin{align}\label{eq:opt}
    \phi(\mu)^\top x&=\min_{(x,u)\in\mathcal{P}}\mu^\top u
  \end{align}
  for all $(x,v)\in\mathcal{P}$, $(\eta,\mu)\in\mathcal{P}^*$.
  Then the following three statements are equivalent:\\[-1mm]
  \begin{itemize}
    \item[\itemi] The problem \cref{eqn:problem} has a finite value for all $(x_0,u_0)\in\mathcal{P}$.\\
    \item[\ii] There exists a $(\lambda_*,0)\in\mathcal{P}^*$ such that
    \[
      \lambda_*=s+A^\top\lambda_*+\phi(r+B^\top \lambda_*).
    \]
    \item[\iii] The convex program 
    \begin{align*}
      \text{Maximize }&\lambda^\top x_0
      \text{ over }(\lambda,0)\in\mathcal{P}^*\\ 
      \text{subject to }&(s+A^\top\lambda-\lambda,r+B^\top \lambda)\in\mathcal{P}^*.
    \end{align*}
    has a bounded value for every $(x_0,u_0)\in\mathcal{P}$.
  \end{itemize}
  If {\ii} is true, then the minimal value of \itemi{} and the maximal value of {\iii} are both equal to $\lambda_*^\top x_0$. Moreover, the optimal control law $x\mapsto u$ is obtained by minimizing $(r+B^\top \lambda_*)^\top u$ over $u$ subject to the constraint $(x,u)\in\mathcal{P}$.
\end{thm}

\begin{proof}
  Given $(\lambda_0,0)\in\mathcal{P}^*$, consider the Bellman recursion 
  \begin{align*}
    J_{k+1}(x)&=\min_{x,u)\in\mathcal{P}}\left[s^\top x+r^\top u+J_k(Ax+Bu)\right],\\
    J_0(x)&=\lambda_0^\top x
  \end{align*}
  and the sequence $\{\lambda_k\}_{k=0}^\infty$ generated by
  \begin{align}
    \lambda_{k+1}&=s+A^\top \lambda_k+\phi(r+B^\top\lambda_k).
    \label{eqn:lambdaiter}
  \end{align}
  We will prove by induction that this implies 
  $J_k(x)=\lambda_k^\top x$ for all $k$:
  Assume for some $k\ge0$ that 
  $J_k(x)=\lambda_k^\top x$ and $(\lambda_k,0)\in\mathcal{P}^*$.  
  Let $(\eta_k,\mu_k):=(s+A^\top \lambda_k,r+B^\top\lambda_k)$. Then
  \begin{align*}
    \eta_k^\top x+\mu_k^\top u&=s^\top x+r^\top u+\lambda_k^\top(Ax+Bu)\ge0
  \end{align*}
  for all $(x,u)\in\mathcal{P}$, so $(\eta_k,\mu_k)\in\mathcal{P}^*$. 
  Hence
  \begin{align*}
    J_{k+1}(x)&=\min_{x,u)\in\mathcal{P}}\left[s^\top x+r^\top u+J_k(Ax+Bu)\right]\\
    &=\min_{(x,u)\in\mathcal{P}}\left[s^\top x+r^\top u+\lambda_k^\top(Ax+Bu)\right]\\
    &=s^\top x+\lambda_k^\top Ax+\min_{(x,u)\in\mathcal{P}}\left[\mu_k^\top u\right]\\
    &=\left[s+A^\top \lambda_k+\phi(\mu_k)\right]^\top x
    =\lambda_{k+1}^\top x
  \end{align*}
  and $\lambda_{k+1}^\top x\ge 0$, so $(\lambda_{k+1},0)\in\mathcal{P}^*$ and the induction assumption holds also for $k+1$. The assumption is trivial for $k=0$, so it follows by induction for $k=1,2,3,\ldots$
  
  \ii{} $\Longleftrightarrow$ \itemi: If \itemi{} holds, then the Bellman recursion from $J_0=0$ gives $0=J_0\le J_1\le J_2\le\ldots$ with upper limit $\lambda_*^\top x$ where $\lambda_*$ satisfies the condition in {\ii}. Conversely, if {\ii}, 
  then $\lambda_0=0$ gives $(\lambda_*-\lambda_k,0)\in\mathcal{P}^*$ for all $k$. This is trivial for $k=0$ and follows by induction for all other $k$ since 
  \begin{align*}
    \lambda_{k+1}^\top x
    &=\min_{(x,u)\in\mathcal{P}}\left[s^\top x+r^\top u+\lambda_k^\top(Ax+Bu)\right]\\
    &\le\min_{(x,u)\in\mathcal{P}}\left[s^\top x+r^\top u+\lambda_*^\top(Ax+Bu)\right]
    =\lambda_*^\top x.
  \end{align*}
  Hence the $J_k$-sequence has an upper limit and \itemi{} holds too. This proves the implication from {\ii} to \itemi{}, as well as the formula $\lambda_*^\top x_0$ for the minimal value in \itemi{}. Given $x$, the optimal control is $u$ is obtained by the minimization
  \begin{align*}
    \arg\min_{(x,u)\in\mathcal{P}}\left[s^\top x+r^\top u+\lambda_*^\top(Ax+Bu)\right]\
  \end{align*}
  or equivalently $\arg\min_{(x,u)\in\mathcal{P}}\left[(r+B^\top \lambda_*)^\top u\right]$.
  
  \ii$\implies$\iii: Consider $\lambda_*$ satisfying {\ii} and let $\lambda$ be a feasible point for the optimization problem in \iii. In particular,
  \begin{align}
    \lambda^\top x&\le s^\top x+r^\top u+\lambda^\top(Ax+Bu)
    &&\text{for }(x,u)\in\mathcal{P}.
  \label{eqn:Bineq}
  \end{align}
  The assumption that $(s,r)$ is an interior point of $\mathcal{P}^*$ gives existence of $\gamma,\delta>0$ such that 
  \begin{align*}
    \lambda_*^\top(Ax+Bu)&\le\gamma(s^\top x+r^\top u)\\
    \lambda^\top x&\le\delta s^\top x\le\delta \lambda_*^\top x&
    \text{for }(x,u)\in\mathcal{P}.
  \end{align*}
  Let $\{\lambda_k\}_{k=0}^\infty$ be the sequence generated by \cref{eqn:lambdaiter} starting from $\lambda_0=\lambda$.
  Combination of \cref{eqn:lambdaiter} and \cref{eqn:Bineq} gives $(\lambda_k-\lambda,0)\in\mathcal{P}^*$ for all $k$. Moreover, 
  \cite[Proposition~2]{rantzer2006relaxed} implies that
  \begin{align*}
    0&\le \lambda_k^\top x\le\left[1+\frac{\delta-1}{(1+\gamma^{-1})^k}\right]\lambda_*^\top x
    &&\text{for all }k,
  \end{align*}
  so in the limit $(\lambda_*-\lambda,0)\in\mathcal{P}$ and $\lambda_*^\top x$ an upper bound for the value of the convex program in \iii.

  \iii$\implies$\ii: Consider the sequence defined by \cref{eqn:lambdaiter} initialized with $\lambda_0=0$. Then
  \begin{align*}
    \lambda_k^\top x&\le \lambda_{k+1}^\top x
    = s^\top x+r^\top u+\lambda_k^\top(Ax+Bu),
  \end{align*}
  so $\lambda_k$ is a feasible point for the optimization problem in \iii{}. Hence \iii{} implies that the sequence $\lambda_1^\top x\le\lambda_2^\top x\le\ldots$ is bounded for every $x$. Therefore the limit $\lambda_*=\lim_{k\to\infty}\lambda_k$ exists and satisfies \ii{}. The proof is complete.
\end{proof}

\Cref{thm:main} supposes the existence of a function $\phi$ with the key property that it renders the optimization in \cref{eq:opt} linear in $x$. It is this feature that keeps the value iteration linear in $x$, which is instrumental in allowing the recasting of the optimal control problem as either a fixed point equation or convex program (conditions \ii{} and \iii{} in \Cref{thm:main}). It is natural then to ask when such a $\phi$ exists. The following result shows that this can be established in a rather direct way through a property of the dual cone $\mathcal{P}^*$. We will use this result in the next section when we connect \Cref{thm:main} to some well studied optimal control problems.

\begin{thm}\label{thm:minpoint}
Let $\mathcal{P}$ be a proper cone in $\realR^{n+m}$ with dual cone $\mathcal{P}^*$. For any $\mu\in\realR^m$ such that the set
\begin{equation}\label{eq:slater}
\mathcal{C}_\mu=\cfunof{\lambda:\funof{\lambda,\mu}\in\mathcal{P}^*}
\end{equation}
has a non-empty interior, the following are equivalent:
\begin{enumerate}
\item[\itemi] 
The function
\begin{equation}\label{eq:cone1}
x\mapsto\inf_{\funof{x,u}\in\mathcal{P}}\mu^{\top}u
\end{equation}
with domain $\cfunof{x:\funof{x,v}\in\mathcal{P}\text{ for some }v}$ is a linear form (that is, the function above takes the form $f(x)=\phi^{\top}x$ for some $\phi\in\realR^n$).
\item[\ii] The set $\mathcal{C}_\mu$ has a minimum element (that is, there exists a $\bar{\lambda}\in\mathcal{C}_\mu$ such that $\lambda\in\mathcal{C}_\mu\implies\funof{\lambda-\bar{\lambda},0}\in\mathcal{P}^*$).
\end{enumerate}
Furthermore, if $\bar{\lambda}$ is the minimum element in \ii{}, then the linear form in \itemi{} is $f(x)=-\bar{\lambda}^{\top}x$.
\end{thm}

\begin{proof}
\itemi$\implies$\ii: Under the hypothesis of \itemi{}, there exists a $\bar{\lambda}$ such that the function in \cref{eq:cone1} takes the form $f\funof{x}=-\bar{\lambda}^\top{}x$. We will first establish that $
\funof{\bar{\lambda},\mu}\in\mathcal{P}^*$. Assume the converse. There therefore exists an
\[
\funof{\bar{x},\bar{u}}\in\mathcal{P}\;\text{such that}\;\begin{bmatrix}\bar{\lambda}\\\mu\end{bmatrix}^{\top}\begin{bmatrix}\bar{x}\\\bar{u}\end{bmatrix}<0,
\]
which implies that $-\bar{\lambda}^{\top}\bar{x}>\mu^{\top}\bar{u}$. However from \cref{eq:cone1} we see that $-\bar{\lambda}^{\top{}}\bar{x}\leq{}\mu^{\top}\bar{u}$, which is a contradiction. We will now show that $\bar{\lambda}$ is the minimum element of $\mathcal{C}_\mu$. The optimization in \cref{eq:cone1} is a cone program on standard form, with dual program (see for example \cite[p.266]{BV04})
\begin{equation}\label{eq:cone2}
\sup_{\funof{\lambda,\mu}\in\mathcal{P}^*}-\lambda^{\top}x.
\end{equation}
By weak duality this implies that
\[
\inf_{\funof{x,u}\in\mathcal{P}}\mu^{\top}u\geq{}\sup_{\funof{\lambda,\mu}\in\mathcal{P}^*}-\lambda^{\top}x,
\]
and so for any $\funof{x,u}\in\mathcal{P}$,
\begin{equation}\label{eq:contcont1}
\funof{\lambda,\mu}\in\mathcal{P}^*\implies\bar{\lambda}^{\top{}}x\leq{}\lambda^{\top}x.   
\end{equation}
This implies that $\funof{\lambda-\bar{\lambda},0}\in\mathcal{P}^*$, which proves the claim.

\ii$\implies$\itemi{}: For all $\funof{x,u}\in\mathcal{P}$,
\[
\begin{bmatrix}\lambda-\bar{\lambda}\\0\end{bmatrix}^{\top}\begin{bmatrix}x\\u\end{bmatrix}\geq{}0,\implies-\funof{\lambda-\bar{\lambda}}^{\top}x\leq{}0.
\]
Therefore for any such $\lambda$,
\[
-\lambda^{\top}x=-\funof{\lambda-\bar{\lambda}}^{\top}x-\bar{\lambda}^{\top{}}x\leq{}-\bar{\lambda}^{\top{}}x.
\]
It then follows that the solution to \cref{eq:cone2} equals $-\bar{\lambda}^{\top{}}x$. Condition \cref{eq:slater} ensures that \cref{eq:cone2} is strictly feasible. Therefore strong duality holds, which proves \cref{eq:cone1} as required.
\end{proof}

\newpage{}

\section{Examples}

\subsection{The Linear Quadratic Regulator}\label{sec:lqr}

In this example we will connect the standard \gls{lqr} problem to  \Cref{thm:main}. This will give an interpretation of our more abstract notation and results in a hopefully more familiar setting. In particular we will show that when the optimization in \itemi{} corresponds to the \gls{lqr} problem, conditions \ii{} and \iii{} reduce to the algebraic Riccati equations and semi-definite programs that are standard in the solution to this problem.

We will in fact study a slightly more general problem. The connections to the standard \gls{lqr} will be established at the end of the section. The optimization problem that we consider first is:
\begin{align*}     \text{Minimize}\quad&\sum_{t=0}^\infty\tr\funof{\begin{bmatrix}
        S&R_1\\R_1^\top{}&R_2
    \end{bmatrix}\begin{bmatrix}
    X(t)& U_1(t)\\U_1(t)^\top{}&U_2\funof{t}
\end{bmatrix}}\\
      \text{over}\quad&\{U_1(t),U_2(t)\}_{t=0}^\infty,\\       
\text{subject to}\quad& X(t+1) = \begin{bmatrix}F&G\end{bmatrix}
\begin{bmatrix}
    X(t)& U_1(t)\\U_1(t)^\top{}&U_2\funof{t}
\end{bmatrix}
\begin{bmatrix}F^\top{}\\G^{\top}\end{bmatrix},\\
\quad&
\begin{bmatrix}
    X(t)& U_1(t)\\U_1(t)^\top{}&U_2\funof{t}
\end{bmatrix}\succeq{}0,\;X(0)=X_0\succeq{}0.
\end{align*}
In the above, the variables are the matrices $X(t)$, $U_1(t)$ and $U_2(t)$. The problem data is the matrices $S$, $R_1$, $R_2$, $F$ and $G$, where we additionally require that
\begin{equation}\label{eq:costcon}
\begin{bmatrix}
    S&R_1\\R_1^\top{}&R_2
\end{bmatrix}\succ{}0.
\end{equation}
Although it may not be immediately clear at first sight, this is an optimization of precisely the type considered in \Cref{thm:main}\itemi{} where $\mathcal{P}$ is the positive semi-definite cone (which is positive, proper and self-dual). 

We will now explain the correspondence between the variables and data in the above, and those in \Cref{thm:main}\footnote{The set of $p\times{}p$ positive semi-definite matrices is a proper cone in $\realR^{p\funof{p+1}/2}$, and so the mappings required to match the two problems explicitly can certainly be defined. However manipulations on the semi-definite cone are often more conveniently expressed in terms of matrices. We will have a need for such manipulations to identify the function $\phi$, and so proceed with this more informal connection between the two problems.}. The variables are related as follows:
\begin{align*}
x(t)&\longleftrightarrow{}X(t)\\
u(t)&\longleftrightarrow{}\begin{bmatrix}
    0&U_1(t)\\U_1(t)^\top{}&U_2(t)
\end{bmatrix}.
\end{align*}
That is the $x(t)$ and $u(t)$ variables correspond to a partitioning of a larger matrix. For the dynamics we have that:
\begin{align*}
Ax(t)&\longleftrightarrow{}FX(t)F^{\top}\\
Bu(t)&\longleftrightarrow{}\begin{bmatrix}F&G\end{bmatrix}
\begin{bmatrix}
    0& U_1(t)\\U_1(t)^\top{}&U_2\funof{t}
\end{bmatrix}
\begin{bmatrix}F^\top{}\\G^{\top}\end{bmatrix}.
\end{align*}
These dyanamics are clearly cone preserving under the constraint $\funof{x,u}\in\mathcal{P}$, as required by \Cref{thm:main}. Similarly the costs can be connected through:
\begin{align*}
s^{\top}x(t)&\longleftrightarrow{}\tr\funof{SX(t)}\\
r^{\top}u(t)&\longleftrightarrow{}\tr\funof{\begin{bmatrix}0&R_1\\R_1^{\top}&R_2\end{bmatrix}
\begin{bmatrix}
    0& U_1(t)\\U_1(t)^\top{}&U_2\funof{t}
\end{bmatrix}
}.
\end{align*}
Since the semi-definite cone is self-dual, under \cref{eq:costcon} we have that $\funof{s,r}\in\mathcal{P}^*$ as required. We have now established that our matrix optimization problem is on the form of \Cref{thm:main}\itemi{}, and that all the prerequisits of the Theorem are satisfied except for the existence of the function $\phi$. We will now address this using \Cref{thm:minpoint}.

To this end, let us now examine \Cref{thm:minpoint} when $\mathcal{P}$ is the semi-definite cone. Again it is helpful to relate vector and matrix variables, and so introduce
\begin{align*}
\lambda&\longleftrightarrow{}\Lambda\\
\mu&\longleftrightarrow{}\begin{bmatrix}
    0&M_1\\M_1^{\top}&M_2
\end{bmatrix}.
\end{align*}
After specialising the notation to this setting, \Cref{thm:minpoint} then states that the existence of $\phi\funof{\mu}$ is equivalent to the set
\[
\mathcal{C}_\mu=\cfunof{\Lambda:\begin{bmatrix}
    \Lambda&M_1\\M_1^{\top}&M_2
\end{bmatrix}\succeq{}0},
\]
having a minimum element. For the semi-definite cone whenever $\mathcal{C}_\mu$ is non-empty, this element exists, and is given by $\overline{\Lambda}=M_1M_2^+M_1^\top{}$. Therefore
\[
\phi\funof{\mu}\longleftrightarrow{}-M_1M_2^+M_1^\top{}.
\]
We are now ready to examine \Cref{thm:main}\ii{}--\iii{}. Substituting in for all found correspondances, we see that in our matrix notation\footnote{It is perhaps worth noting that from basic properties of adjoints, it follows that
\begin{align*}
A^{\top}\lambda(t)&\longleftrightarrow{}F^{\top}\Lambda(t)F\\
B^{\top}\lambda(t)&\longleftrightarrow{}
\begin{bmatrix}
    0&F^{\top}\Lambda(t)G\\
    G^{\top}\Lambda(t)F&G^{\top}\Lambda(t)G
\end{bmatrix}.
\end{align*}}, the equation $\lambda_*=s+A^\top\lambda_*+\phi(r+B^\top \lambda_*)$ becomes
\begin{align*}
\Lambda_* &= S + F^{\top}\Lambda_*F-\funof{F^{\top}\Lambda_*G+R_1}\funof{R_2+G^\top{}\Lambda_*G}^{-1}
\funof{G^{\top}\Lambda_*F+R_1^\top{}}.
\end{align*}
That is, the equation in \Cref{thm:main}\ii{} becomes the algebraic Riccati equation (the pseudo-inverses become inverses in light of \cref{eq:costcon}). This can then be solved in the usual way whenever the pair $\funof{F,G}$ is stabilizable. \Cref{thm:main}\iii{} can be similarly analyzed. In this case, the convex program reduces to the semi-definite program
\begin{align*}
      \text{Maximize }&\tr\funof{\Lambda{}X_0}
      \\ 
      \text{subject to }&
        \begin{bmatrix}
            S-\Lambda&R_1\\R_1^{\top}&R_2
        \end{bmatrix}+
        \begin{bmatrix}F^\top{}\\G^\top{}
        \end{bmatrix}\Lambda
        \begin{bmatrix}
            F&G        \end{bmatrix}\succeq{}0,\;\Lambda\succeq{}0.
    \end{align*}

To conclude the example, let us finally connect explicitly to the standard \gls{lqr} problem. The standard \gls{lqr} problem is concerned with the optimization
\begin{align*}     \text{Minimize}\quad&\sum_{t=0}^\infty\begin{bmatrix}x(t)\\u(t)\end{bmatrix}^{\top}
    \begin{bmatrix}
        S&R_1\\R_1^\top{}&R_2
    \end{bmatrix}
\begin{bmatrix}x(t)\\u(t)\end{bmatrix}
      \text{ over }\{u(t)\}_{t=0}^\infty\\ 
      \text{subject to}\quad& x(t+1)=Fx(t)+Gu(t),\, x(0)=x_0.
\end{align*}
To see the connection to the matrix optimization only requires us to notice that whenever
\begin{equation}\label{eq:r1}
\begin{bmatrix}
    X(t)& U_1(t)\\U_1(t)^\top{}&U_2\funof{t}
\end{bmatrix} =
\begin{bmatrix}
    x(t)\\u(t)
\end{bmatrix}\begin{bmatrix}
    x(t)^{\top}&u(t)^{\top}
\end{bmatrix},
\end{equation}
the specified matrix dynamics ensure that
\[
X(t+1) = x(t+1)x(t+1)^\top{},
\]
and also that
\[
\begin{aligned}
\tr\funof{\begin{bmatrix}
        S&R_1\\R_1^\top{}&R_2
    \end{bmatrix}\begin{bmatrix}
    X(t)& U_1(t)\\U_1(t)^\top{}&U_2\funof{t}
\end{bmatrix}}=
\begin{bmatrix}x(t)\\u(t)\end{bmatrix}^{\top}
    \begin{bmatrix}
        S&R_1\\R_1^\top{}&R_2
    \end{bmatrix}
\begin{bmatrix}x(t)\\u(t)\end{bmatrix}.
\end{aligned}
\]
That is, under \cref{eq:r1}, the costs and constraints in the standard \gls{lqr} problem, and the matrix optimization problem, are identical. It then follows (with a little extra work) that when the initial condition $X_0$ is rank 1, that the studied matrix optimization collapses to the standard \gls{lqr} problem.

\subsection{A Linear Regulator on a Polyhedral Cone}

In this subsection we will recover the results of \cite[Theorem 1]{rantzer2022explicit}. Consider the application of \Cref{thm:main} when $\mathcal{P}$ is the polyhedral cone
\[
\mathcal{P}=\cfunof{\begin{bmatrix}
    x\\u
\end{bmatrix}:\begin{bmatrix}
    I&0\\E&-I\\E&I
\end{bmatrix}
\begin{bmatrix}
    x\\u
\end{bmatrix}\geq{}0},
\]
where $E\geq{}0$. The motivation for this choice comes from the study of positive systems. Note in particular that the condition $(x,u)\in\mathcal{P}$ constrains the system state to lie in the positive orthant, and the input to satisfy
\[
\abs{u}\leq{}Ex.
\]
When minimizing the cost in \Cref{thm:main}{\itemi}, we are in effect optimizing performance under the constraint that the control keeps the system state positive.

First note that the dual cone $\mathcal{P}^*$ is also polyhedral, and is given by
\begin{equation}\label{eq:polydual}
\mathcal{P}^*=\cfunof{
\begin{bmatrix}
    \lambda\\\mu
\end{bmatrix}:\begin{bmatrix}
    \lambda\\\mu
\end{bmatrix} = 
\begin{bmatrix}
    I&0\\
    E&-I\\
    E&I
\end{bmatrix}^\top
\begin{bmatrix}
    w\\y\\z
\end{bmatrix},\begin{bmatrix}
    w\\y\\z
\end{bmatrix}\geq{}0
}.
\end{equation}
The conditions of \Cref{thm:main} then provide conditions for solving the optimization in \itemi{} whenever:
\begin{enumerate}
    \item[1)] For all $\funof{x,u}\in\mathcal{P}$, there exists a $v$ such that $\funof{Ax+Bu,v}\in\mathcal{P}$.
    \item[2)] $\funof{s,r}\in\mathcal{P}^*$.
    \item[3)] There exists a suitable function $\phi$.
\end{enumerate}
It is readily checked that 1) amounts to requiring that $A\geq{}\abs{EB}$. For 2) and 3) it is convenient to appeal to \Cref{thm:minpoint}. From \cref{eq:polydual} we see that for any $\mu$, $y$ and $z$ can be chosen according to
\[
y=\mu_-+v,\;z=\mu_++v,
\]
where $v\geq{}0$, and the $\mu_-$ and $\mu_+$ denote the positive and negative parts of the vector $\mu$ respectively. It then follows that $\lambda\in\mathcal{C}_\mu$ if and only if
\[
\lambda = w + E^{\top}\funof{\abs{\mu}+v},
\]
where $w\geq{}0$. The set in \Cref{thm:minpoint}\ii{} therefore has a minimum element given by $\overline{\lambda}=E^{\top}\abs{\mu}$. This implies that for (2) we require that $s>E^{\top}\abs{r}$, and also that
\[
\phi\funof{\mu}=-E^{\top}\abs{\mu}.
\]
Condition \ii{} in \Cref{thm:main} then collapses to the existence of a $\lambda_*\geq{}0$ such that
    \begin{align*}
      \lambda_*=s+A^\top\lambda_*-E^{\top}\abs{r+B^\top \lambda_*},
    \end{align*}
which can be checked with linear programming. The convex program in \iii{} similarly simplifies, this time to the following linear program:
\begin{align*}
\text{Maximize }&\lambda^\top x_0\\
      \text{subject to }&
      \begin{bmatrix}s\\r\end{bmatrix}=
      \begin{bmatrix}I-A^{\top}&I&E^{\top}&E^{\top}\\-B^{\top}&0&-I&-I
      \end{bmatrix}
      \begin{bmatrix}
      \lambda\\w\\y\\z
      \end{bmatrix}\!,\begin{bmatrix}
      \lambda\\w\\y\\z
      \end{bmatrix}\!\!\geq{}\!0.
\end{align*}

\subsection{A Structured Linear Quadratic Regulation Problem}

In this example we will continue our study of the matrix \gls{lqr} problem outlined at the beginning of \cref{sec:lqr}. However, we now consider the case that the dynamics preserve a cone $\mathcal{P}$ that contains the semi-definite cone. When considering the application of \Cref{thm:main}, this will implicitly add further constraints to the allowable dynamics. However the upshot is that the optimal control will inherit these structural constraints, meaning that they can be exploited when implementing the optimal control laws. As we will see, this shows for example that the optimal control laws (under suitable definitions) for systems with a circulant structure inherit the same circulant structure \cite{BPD02,BJM12}, with a host of other extensions, for example to systems defined through a generalized frequency variable \cite{HHS09}.

The ideas that we are about to present can be significantly generalized, but should serve to illustrate the core ideas. We start with the definition of the cone.

\begin{defn}\label{def:1}
    Let $Q\in\realR^{m\times{}m}$ be invertible, and
    \[
    T = \begin{bmatrix}Q\otimes{}I_p&0\\0&Q\otimes{}I_q\end{bmatrix}.
    \]
    Define
    {\small\begin{align*}
      \mathcal{P}_Q=
        \cfunof{\begin{bmatrix}
            X&U_1\\U_1^\top{}&U_2
        \end{bmatrix}\in\hermR^{m\funof{p+q}}:
        A\funof{T
        \begin{bmatrix}
            X&U_1\\U_1^\top{}&U_2
        \end{bmatrix}T^\top{}
        }\succeq{}0},
    \end{align*}
    }where $\hermR^{m\funof{p+q}}$ denotes the set of ${m\funof{p+q}}\times{}{m\funof{p+q}}$ symmetric matrices,
    \[
    A\funof{
    \begin{bmatrix}
        W&X\\Y&Z
    \end{bmatrix}}=\begin{bmatrix}
        \diag_m\funof{W}&\diag_m\funof{X}\\
        \diag_m\funof{Y}&\diag_m\funof{Z}
    \end{bmatrix},
    \]
    and $\diag_m:\realR^{mx\times{}my}\rightarrow{}\realR^{mx\times{}my}$ denotes the operator that zeros out a given matrix outside a set of $m$ blocks of size $x\times y$ along the diagonal\footnote{That is 
    \[
\diag_m\funof{M}=\begin{bmatrix}M_{11}&0&\cdots{}&0
\\0&M_{22}&0
\\\vdots{}&0&\ddots{}&0
\\0&\cdots{}&0&M_{mm}
\end{bmatrix},M_{kk}\in\realR^{x\times{}y}.
\]}.
\end{defn}

The following lemma establishes the key properties of the cone $\mathcal{P}_Q$ that we will require to apply \Cref{thm:main}. The first part of the result gives a formula for the dual cone, and establishes that a suitable function $\phi\funof{\mu}$ exists (the function is in fact the same as that from \cref{sec:lqr}). The final part gives a formula that can be used to compute the optimal control law, and it is from this formula that we can deduce that the optimal control inherits additional structural features as claimed above.

\begin{lem}\label{lem:1}
Let $A$, $T$, $\mathcal{P}_Q$ and all the block partitioning of matrices be as in \Cref{def:1}. Then
\[
\mathcal{P}_Q^*=\cfunof{T^{\top}A\funof{Y}T:Y\succeq{}0}.
\]
Furthermore if there exists a $\Lambda$ such that
\begin{equation}\label{eq:completion}
\begin{bmatrix}
     \Lambda&M_1\\M_1^{\top}&M_2
\end{bmatrix}\in\mathcal{P}^*_Q,
\end{equation}
then for fixed $X$
\begin{equation}\label{eq:reqop}
\inf\!\cfunof{\!\tr\!\funof{\!
\begin{bmatrix}
    0&M_1\\M_1^{\top}&M_2
\end{bmatrix}\!\!
\begin{bmatrix}
    0&U_1\\U_1^{\top}&U_2
\end{bmatrix}\!}
\!:\!\begin{bmatrix}
    X&U_1\\U_1^{\top}&U_2
\end{bmatrix}\!\in\mathcal{P}_Q\!}
\end{equation}
equals $\tr\funof{M_1M_2^+M_1^{\top}X}$. A pair $\funof{U_1,U_2}$ that achieves the infimum above can be computed through the equation
\begin{equation}\label{eq:optconpq}
\begin{bmatrix}
    X&U_1\\U_1^{\top}&U_2
\end{bmatrix}=
\begin{bmatrix}
I\\-M_2^+M_1^\top
\end{bmatrix}
X
\begin{bmatrix}
I\\-M_2^+M_1^\top
\end{bmatrix}^\top{}.
\end{equation}
\end{lem}
\vspace{.2cm}
\begin{proof}
See \Cref{app:1}.
\end{proof}

To understand how to leverage \Cref{lem:1}, consider again the optimal control problem from \cref{sec:lqr}, but suppose that the matrices $F$ and $G$ have a block-wise circulant structure. That is,
\[
F = \begin{bmatrix}
    f_1\bar{F}&f_2\bar{F}&\cdots{}&f_m\bar{F}\\
    f_m\bar{F}&f_1\bar{F}&\cdots{}&f_{m-1}\bar{F}\\
    \vdots{}&\vdots{}&\ddots{}&\vdots{}\\
    f_2\bar{F}&f_3\bar{F}&\cdots{}&f_1\bar{F}
\end{bmatrix},
\]
with a similar expression for $G$. We may write this compactly according to
\begin{equation}\label{eq:simdiag}
F = F_{\text{sd}}\otimes{}\bar{F}\;\;\text{and}\;\;G=G_{\text{sd}}\otimes{}\bar{G},
\end{equation}
where $\bar{F}\in\realR^{p\times{}p}$ and $\bar{G}\in\realR^{p\times{}q}$, and
\[
\!F_{\text{sd}}\!=\!\!
    \begin{bmatrix}
    f_1&f_2&\cdots{}&f_m\\
    f_m&f_1&\cdots{}&f_{m-1}\\
    \vdots{}&\vdots{}&\ddots{}&\vdots{}\\
    f_2&f_3&\cdots{}&f_1
\end{bmatrix}\!\!,
G_{\text{sd}}\!=\!\!
    \begin{bmatrix}
    g_1&g_2&\cdots{}&g_m\\
    g_m&g_1&\cdots{}&g_{m-1}\\
    \vdots{}&\vdots{}&\ddots{}&\vdots{}\\
    g_2&g_3&\cdots{}&g_1
\end{bmatrix}\!\!.
\]
Critically for our purposes, the matrices $F_{\text{sd}}$ and $G_{\text{sd}}$ are simultaneously diagonalizable. That is, there exists an invertible matrix $Q$ such that both
\[
Q^{-1}F_{\text{sd}}Q\;\;\text{and}\;\;Q^{-1}G_{\text{sd}}Q
\]
are diagonal. This implies that
\[
\funof{Q^{-1}\!\otimes{}I_p}F\funof{Q\otimes{}I_p}\!=\!\diag_m\funof{\funof{Q^{-1}\!\otimes{}I_p}F\funof{Q\otimes{}I_p}}
\]
and 
\[
\funof{Q^{-1}\!\otimes{}I_p}G\funof{Q\otimes{}I_q}\!=\!\diag_m\funof{\funof{Q^{-1}\!\otimes{}I_p}G\funof{Q\otimes{}I_q}}.
\]
It follows from these expressions that for such matrices,
\[
\begin{bmatrix}
    X(t)& U_1(t)\\U_1(t)^\top{}&U_2\funof{t}
\end{bmatrix}\in\mathcal{P}_Q\implies
\begin{bmatrix}
    X(t+1)& 0\\0&0
\end{bmatrix}\in\mathcal{P}_Q,
\]
where
\[
X(t+1) = \begin{bmatrix}F&G\end{bmatrix}
\begin{bmatrix}
    X(t)& U_1(t)\\U_1(t)^\top{}&U_2\funof{t}
\end{bmatrix}
\begin{bmatrix}F^\top{}\\G^{\top}\end{bmatrix}.
\]
Taking an analogous approach to that in \cref{sec:lqr}, we see that under suitable assumptions on the cost function, \Cref{thm:main} can be applied to solve a relaxed version of the standard \gls{lqr} problem, where the state and inputs can be drawn from $\mathcal{P}_Q$. Since $\mathcal{P}_Q$ contains the positive semi-definite cone, the solution to this problem will always give a lower bound on the achievable cost to the standard \gls{lqr} problem. However more is true. In particular we see from the final part of \Cref{thm:main} and \cref{eq:optconpq} in \Cref{lem:1} that the optimal inputs to the relaxed problem can be computed through the control law $X\mapsto{}\funof{U_1,U_2}$ defined implicitly by
\[
\begin{bmatrix}
    X&U_1\\U_1^{\top}&U_2
\end{bmatrix}=
\begin{bmatrix}
I\\-K
\end{bmatrix}
X
\begin{bmatrix}
I\\-K
\end{bmatrix}^\top{},
\]
where
\[
K=\funof{R_2+G^\top{}\Lambda_*G}^{-1}
\funof{G^{\top}\Lambda_*F+R_1^\top{}}.
\]
It follows from these expressions that whenever the initial condition satisfies $X\funof{0}\succeq{}0$,
\[
\begin{bmatrix}
    X\funof{t}&U_1\funof{t}\\U_1\funof{t}^{\top}&U_2\funof{t}
\end{bmatrix}\succeq{}0
\]
for all $t$. Hence the found controls are also feasible, and hence also optimal, for the standard \gls{lqr} problem. It is then easily seen from the expression for the dual cone that
\[
\funof{Q^{-1}\!\otimes{}I_q}K\funof{Q\otimes{}I_p}\!=\!\diag_m\funof{\funof{Q^{-1}\!\otimes{}I_q}K\funof{Q\otimes{}I_p}}.
\]
Therefore the optimal control law additionally inherits the block-wise circulant structure present in the problem data. Interestingly this cone based treatment makes no use of any orthogonality properties of the matrix $Q$ allowing similar insights to be derived for general dynamics that take the form of \cref{eq:simdiag}, see for example \cite{HHS09}.

\section{Conclusions}

We have proved a general result on optimal control on positive cones, covering standard linear quadratic control and corresponding results for positive systems as special cases. Of course, many other special cases can be derived by defining other relevant cones, or combinations of the ones above. Hopefully, this will in the future form the basis for a powerful toolbox combining the versatility of linear quadratic control with the rich scalability properties of positive systems and shortest path problems in networks. Preliminary steps in this direction have been taken in \cite{gurpegui2023minimax} and \cite{ohlin2023optimal}.

\section{Acknowledgment}The authors are grateful to Yuchao Li for pointing out an error in the original statement of the main result.

\bibliographystyle{IEEEtran}
\bibliography{references.bib}

\appendix

\section{Proof of \Cref{lem:1}}\label{app:1}

\begin{proof}
We start by developing the formula for the dual cone. Let $B\funof{X} = A\funof{TXT^{\top}}$. By Farkas' lemma,
\[
\mathcal{P}^*_Q=\cfunof{B^*\funof{Y}:Y\succeq{}0}.
\]
Since for any symmetric matrices $C$ and $D$, $\tr\funof{CD}=\vc\funof{C}^{\top}\vc\funof{D}$, we see that
\[
\begin{aligned}
\tr\funof{B\funof{X}Y}&=\tr\funof{A\funof{TXT^\top{}}A\funof{Y}}\\
&=\tr\funof{TXT^{\top}A\funof{Y}}\\
&=\tr\funof{XT^{\top}A\funof{Y}T}.
\end{aligned}
\]
This shows that $B^*\funof{Y}=T^{\top}A\funof{Y}T$ as required. To show the second claim observe that if
\[
        \begin{bmatrix}
            \Lambda&M_1\\M_1^{\top}&M_2
        \end{bmatrix}\in\mathcal{P}^*,
        \]
then there exists a
\[
\begin{bmatrix}Y_{1}&Y_2\\Y_2^{\top}&Y_3\end{bmatrix}=
\begin{bmatrix}\diag_m\funof{Y_1}&\diag_m\funof{Y_2}\\\diag_m\funof{Y_2^{\top}}&\diag_m\funof{Y_3}\end{bmatrix}\succeq{}0
\]
such that
\[
\begin{bmatrix}
    0&M_1\\M_1^{\top}&M_2
\end{bmatrix}=T^{\top}
\begin{bmatrix}0&Y_1\\Y_1^{\top}&Y_2\end{bmatrix}
T.
\]
It is then straightforward to show that
\[
\begin{bmatrix}
    M_1M_2^+M_1^{\top}&M_1\\M_1^{\top}&M_2
\end{bmatrix}=T^{\top}
\begin{bmatrix}Y_2Y_{2}^+Y_1^{\top}&Y_1\\Y_1^{\top}&Y_2\end{bmatrix}
T,
\]
which is in $\mathcal{P}^*_Q$. Since $\mathcal{P}_Q^*$ is contained in the semi-definite cone, $M_1M_2^+M_1^\top$ is the minimum element of the set
\[
\mathcal{C}_\mu=\cfunof{\Lambda:\begin{bmatrix}
     \Lambda&M_1\\M_1^{\top}&M_2
\end{bmatrix}\in\mathcal{P}^*_Q},
\]
and so the formula for the infimum follows by \Cref{thm:minpoint}. In order to verify the equation for the minimizer, first observe that from the above,
\[
\begin{aligned}
&A\funof{T\begin{bmatrix}
I\\-M_2^+M_1^\top
\end{bmatrix}
X
\begin{bmatrix}
I\\-M_2^+M_1^\top
\end{bmatrix}^\top{}
T^\top}=\\
&\quad\begin{bmatrix}
I&0\\-Y_2^+Y_1^\top{}&0
\end{bmatrix}
A\funof{T\begin{bmatrix}
    X&U_1\\U_1^{\top}&U_2
\end{bmatrix}T^{\top}}
\begin{bmatrix}
I&0\\-Y_2^+Y_1^\top{}&0
\end{bmatrix}^{\top}\!\!\!\!.
\end{aligned}
\]
It then follows that
\[
\begin{bmatrix}
    X&U_1\\U_1^{\top}&U_2
\end{bmatrix}\in\mathcal{P}_Q
\implies
\begin{bmatrix}
I\\-M_2^+M_1^\top
\end{bmatrix}
X
\begin{bmatrix}
I\\-M_2^+M_1^\top
\end{bmatrix}^\top{}\!\!\!\!\!\!\in\mathcal{P}_Q.
\]
Next observe that since $\funof{I-M_2^+M_2}M_1^{\top}=0$ (else there would exist no $\Lambda$ such that \cref{eq:completion} holds),
\begin{equation}\label{eq:zero}
\begin{bmatrix}
    M_1M_2^+M_1^\top\!\!&M_1\\M_1^{\top}&M_2
\end{bmatrix}\!\begin{bmatrix}
I\\-M_2^+M_1^\top
\end{bmatrix}
\!X\!
\begin{bmatrix}
I\\-M_2^+M_1^\top
\end{bmatrix}^\top{}\!\!\!\!\!\!=0.
\end{equation}
Finally for fixed $X$ note that the minimizer for \cref{eq:reqop} equals the minimizer for
\[
\inf\!\cfunof{\!\tr\!\funof{\!
\begin{bmatrix}
    M_1M_2^+M_1^\top\!\!\!&M_1\\M_1^{\top}&M_2
\end{bmatrix}\!\!
\begin{bmatrix}
    X&U_1\\U_1^{\top}&U_2
\end{bmatrix}\!}
\!:\!\begin{bmatrix}
    X&U_1\\U_1^{\top}&U_2
\end{bmatrix}\!\in\mathcal{P}_Q\!}\!.
\]
From the definition of a dual cone,
\[
\tr\!\funof{\!
\begin{bmatrix}
    M_1M_2^+M_1^\top&M_1\\M_1^{\top}&M_2
\end{bmatrix}
\begin{bmatrix}
    X&U_1\\U_1^{\top}&U_2
\end{bmatrix}}\geq{}0.
\]
Therefore by \cref{eq:zero} the right hand side of \cref{eq:optconpq} is a minimizer as required.
\end{proof}

\end{document}